\documentclass[letterpaper,10pt]{article}  

\pdfoutput=1 
 
\usepackage[utf8]{inputenc}
\usepackage{amsthm} 
\usepackage[fleqn]{amsmath} 
\usepackage{nccmath} 
\usepackage{amssymb} 
\usepackage[dvipsnames]{xcolor} 
\usepackage{algorithm,algorithmic}
\usepackage[hyphens]{url}

\usepackage[margin=1in]{geometry} 

\usepackage{graphicx} 
\usepackage{subcaption} 

\usepackage{makecell} 
\usepackage{multirow} 
\usepackage{tabu} 
\usepackage{cite}

\newcommand\numberthis{\addtocounter{equation}{1}\tag{\theequation}} 

\newcommand{\Rbb}{\mathbb{R}} 
\newcommand{\Ebb}{\mathbb{E}} 
\newcommand{\Gcal}{\mathcal{G}} 
\newcommand{\Vcal}{\mathcal{V}} 
\newcommand{\Ecal}{\mathcal{E}} 
\newcommand{\Ncal}{\mathcal{N}} 
\newcommand{\Scal}{\mathcal{S}} 
\newcommand{\Ucal}{\mathcal{U}} 

\newcommand{\ones}{\mathbf{1}} 
\newcommand{\zeros}{\mathbf{0}} 
\newcommand{\Exp}[1]{\Ebb{\left[#1\right]}} 

\newcommand{\range}{\text{range}}

\newcommand\norm[1]{\left\lVert#1\right\rVert} 
\newcommand\normtwo[1]{\left\lVert#1\right\rVert_2} 
\newcommand\abs[1]{\left|#1\right|} 
\newcommand\paren[1]{\left(#1\right)} 
\newcommand\Sbraces[1]{\left[#1\right]} 
\newcommand\Cbraces[1]{\left\lbrace#1\right\rbrace} 
\newcommand\Abraces[1]{\left\langle#1\right\rangle} 

\DeclareMathOperator*{\argmax}{argmax} 

\DeclareMathOperator*{\maximize}{maximize}
\newcommand{\st}{\text{subject to}}

\newcommand{\ram}{\rightarrow} 

\newcommand{\algindent}{\hspace{\algorithmicindent}} 

\newcommand{\MYSTATE}{\vspace{0.2em}\STATE} 

\renewcommand{\AA}{{A^+ A}} 
\newcommand{\LL}{{\Lambda^+ \Lambda}} 
\newcommand\normAA[1]{\left\lVert#1\right\rVert_A} 
\newcommand{\SM}{{\text{SM}}} 
\newcommand\normSM[1]{\left\lVert#1\right\rVert_{\SM}} 
\newcommand{\SMNO}{{\text{SMNO}}} 
\newcommand\normSMNO[1]{\left\lVert#1\right\rVert_{\SMNO}} 
\newcommand{\sign}{\text{sign}}

\newcommand{\ts}{\textsuperscript} 

\newtheoremstyle{mytheoremstyle} 
    {\topsep}                    
    {\topsep}                    
    {\normalfont}                   
    {}                           
    {\bfseries}                   
    {.}                          
    {.5em}                       
    {}              
\theoremstyle{mytheoremstyle}

\newtheorem{theorem}{Theorem}

\newtheorem{definition}{Definition}
\newtheorem{lemma}{Lemma}

\newtheorem{remark}{Remark}

\title{Pick your Neighbor: Local Gauss-Southwell Rule for Fast Asynchronous Decentralized Optimization}
\author{Marina Costantini$^{1,*}$ \and Nikolaos Liakopoulos$^2$ \and Panayotis Mertikopoulos$^3$ \and Thrasyvoulos Spyropoulos$^{1,4}$}
\date{\normalsize{
    $^1$EURECOM, Sophia Antipolis\\%
    $^2$Amazon, Luxembourg City\\
    $^3$Univ. Grenoble Alpes, CNRS, Inria, Grenoble INP, LIG \& Criteo AI Lab\\
    $^4$Technical University of Crete\\
    $^*$Correspondence: marina.costantini@eurecom.fr } }

\begin{document}

\maketitle

\begin{abstract}

    In decentralized optimization environments, each agent $i$ in a network of $n$ nodes has its own private function $f_i$, and nodes communicate with their neighbors to cooperatively minimize the aggregate objective $\sum_{i=1}^n f_i$.
    In this setting, synchronizing the nodes' updates incurs significant communication overhead and computational costs, so much of the recent literature has focused on the analysis and design of \textit{asynchronous} optimization algorithms, where agents activate and communicate at arbitrary times without needing a global synchronization enforcer.
    However, most works assume that when a node activates, it selects the neighbor to contact based on a fixed probability (e.g., uniformly at random), a choice that ignores the optimization landscape at the moment of activation.
    Instead, in this work we introduce an optimization-aware selection rule that chooses the neighbor providing the highest \emph{dual cost improvement} (a quantity related to a dualization of the problem based on consensus).
    This scheme is related to the coordinate descent (CD) method with the Gauss-Southwell (GS) rule for coordinate updates;
    in our setting however, \textit{only a subset of coordinates is accessible at each iteration} (because each node can communicate only with its neighbors), so the existing literature on GS methods does not apply.
    To overcome this difficulty, we develop a new analytical framework for~smooth and strongly convex $f_i$ that covers the class of \emph{set-wise CD} algorithms --a class that directly applies to decentralized scenarios, but is not limited to them-- and we show that the proposed set-wise GS rule achieves a speedup factor of up to the maximum degree in the network (which is in the order of $\Theta(n)$ for highly connected graphs).
    The speedup predicted by our analysis is validated in numerical experiments with synthetic data.

\end{abstract}

\section{Introduction}

    Many timely applications require solving optimization problems over a network where nodes can only communicate with their direct neighbors.
    This may be due to the need of distributing storage and computation loads (e.g. training large machine learning models \cite{lian2017can}), or to avoid transferring data that is naturally collected in a decentralized manner, either due to the communication costs or to privacy reasons (e.g. sensor networks \cite{wan2009event}, edge computing \cite{alrowaily2018secure}). 

    Specifically, we consider a setting where the nodes want to solve the decentralized optimization problem
    \begin{ceqn}
    \begin{equation} \label{eq:primal_unconstrained}
        \underset{\theta \in \Rbb^d}{\text{minimize}} \quad \sum_{i=1}^n f_i(\theta),
    \end{equation}
    \end{ceqn}
    where each local function $f_i$ is known only by node $i$ and nodes can exchange optimization values (parameters, gradients) but \emph{not} the local functions themselves.
    We represent the communication network as a graph $\Gcal = (\Vcal,\Ecal)$ with $n = |\Vcal|$ nodes (agents) and $E = |\Ecal|$ edges, which are the links used by the nodes to communicate with their neighbors. 
    
    Problem \eqref{eq:primal_unconstrained} was formally introduced in \cite{nedic2007rate} and widely studied ever since.
    A convenient reformulation often adopted in the literature assigns to each node a local variable $\theta_i$ and forces consensus between node pairs connected by an edge:
    
    \begin{ceqn}
    \begin{subequations} \label{eq:primal_constrained}
        \begin{alignat}{2} 
        & \underset{\theta_1,\ldots,\theta_n \in R^d}{\text{minimize}} \quad && \sum_{i=1}^n f_i(\theta_i) \label{eq:problem-a} \\ 
        & \st \quad && \theta_i = \theta_j \quad \forall \; (i,j) \equiv \ell \in \mathcal{E}, \label{eq:problem-b}
        \end{alignat} 
    \end{subequations}
    \end{ceqn}
    where $\ell \equiv (i,j)$ indicates that edge $\ell \in \Ecal$ links nodes $i$~and~$j$.
    Decentralized algorithms to solve \eqref{eq:primal_constrained} allow all nodes to~find the minimum value of \eqref{eq:primal_unconstrained} by just communicating with their neighbors and updating their local variables.
    This is in contrast with broadcast AllReduce algorithms \cite{rabenseifner2004optimization} or parallel distributed architectures \cite{xiao2019dscovr}, which were recently shown to be slower than decentralized schemes in some scenarios \cite{lian2017can}.
    
    Here we use reformulation \eqref{eq:primal_constrained} to propose an \emph{asynchronous} decentralized algorithm where nodes activate at any time uniformly at random, and once activated they choose one of their neighbors to make an update.
    Methods with such minimal coordination requirements avoid incurring extra costs of synchronization  that may also slow down convergence, which is the reason why many algorithms for this asynchronous setting have been proposed in the literature \cite{iutzeler2013asynchronous, wei2013convergence, xu2017convergence, pu2020push, srivastava2011distributed, ram2009asynchronous}.
    However, most of these works assume that when a node activates, it simply selects the neighbor to contact randomly, based on a predefined probability distribution. 
    This approach overlooks the possibility of letting nodes \emph{choose} the neighbor to contact taking into account the optimization landscape at the time of activation.
    Therefore, here we depart from the probabilistic choice and ask:
    \emph{can nodes pick the neighbor smartly to make the optimization process converge faster?}

    In this paper we give an affirmative answer and propose an algorithm that achieves this by solving the dual problem of \eqref{eq:primal_constrained}. 
    In the dual formulation, there is one dual variable $\lambda_\ell \in \Rbb^d$ per constraint $\theta_i = \theta_j$, hence
    each dual variable can be associated with an edge $\ell$ in the graph. 
    Our algorithm lets~an activated node $i$ contact a neighbor $j$ so that together they update their shared variable $\lambda_\ell$ with a gradient step.
    In particular, we propose to select the neighbor $j$ such that the updated $\lambda_\ell$ is the one \emph{whose directional gradient for the dual function is the largest}, and thus the one that provides the greatest cost improvement at that iteration.
    Such optimal choice for asynchronous decentralized optimization has not yet been considered in the literature.
    
    Interestingly, the above protocol where a node activates and selects a $\lambda_\ell$ to update can be seen as applying the coordinate descent (CD) method \cite{nesterov2012efficiency} to solve the dual problem of \eqref{eq:primal_constrained}, with the following key difference: unlike standard CD methods, where \emph{any} of the $d$ coordinates may be updated, now \emph{only a small subset of coordinates are accessible at each step}, which are the coordinates associated with the edges connected to the node activated.
    Moreover, our proposal of updating the $\lambda_\ell$ with the largest gradient is similar to the Gauss-Southwell (GS) rule\cite{nutini2015coordinate}, but applied \emph{only} to the parameters accessible by the activated node.
   
    We name such protocols \emph{set-wise CD} algorithms, and we analyze both random uniform and GS coordinate selection within the accessible set.
    To the best of our knowledge, set-wise CD schemes have not yet been explored in the decentralized setting, and thus the speedup achievable by the GS rule compared to uniform sampling is not known.
    Furthermore, three difficulties complicate the analysis and constitute the base of our contributions:
    (i) for arbitrary graphs, the dual problem of \eqref{eq:primal_constrained} has an objective function that is \emph{not} strongly convex, even if the primal functions $f_i$ \emph{are} strongly convex,
    (ii) the fact that the GS rule is applied to a few coordinates prevents the use of standard norms to obtain the linear rate, as commonly done for CD methods \cite{nesterov2012efficiency, nutini2015coordinate, nutini2017let}, and 
    (iii) the coordinate sets are overlapping (i.e. non-disjoint), which makes the problem even harder.

    For this reason, we develop a methodology where we prove strong convexity in norms uniquely defined for each algorithm considered.
    In particular, for the set-wise GS rule this requires relating the norm that we originally define to an alternative norm that considers non-overlapping sets, for which the problem becomes easier and solvable analytically. 

    Finally, our results also apply to the parallel distributed setting where the parameter vector is stored at a single server and workers can update different subsets of its entries \cite{tsitsiklis1986distributed, peng2016arock, xiao2019dscovr}. We show an example in our simulations.
    
    Our contributions can be summarized as follows: 
    
    \begin{itemize} 
        \item We introduce the class of \textit{set-wise CD} algorithms and analyze two variants to pick the coordinate to update in the activated set: one that uses uniform sampling (SU-CD), and another that applies the GS rule (SGS-CD).
        
        \item We show that this class of algorithms can be used to solve \eqref{eq:primal_constrained} asynchronously, and we provide the linear convergence rates of the two variants considered when the primal functions $f_i$ are smooth and strongly convex.
        
        \item To obtain these rates for SU-CD and SGS-CD, we prove strong convexity in uniquely-defined norms that, respectively
        (i) take into account the graph structure to show strong convexity in the linear subspace where the coordinate updates are applied, and
        (ii) account for both the random uniform node activation and the application of the GS rule to just a subset of the coordinates.

        \item We show that the speedup of SGS-CD with respect to SU-CD can be up to $N_{\max}$ (the size of the largest coordinate set), which is analogous to the that of the GS rule with respect to random uniform coordinate sampling in centralized CD \cite{nutini2015coordinate}. 
    \end{itemize}

\section{Related work} \label{sec:related_work}

    A number of algorithms have been proposed to solve \eqref{eq:primal_unconstrained} asynchronously.
    In \cite{ram2009asynchronous}, the activated node chooses a neighbor uniformly at random and both nodes average their primal local values.
    In \cite{iutzeler2013asynchronous} the authors adapted the ADMM algorithm to the decentralized setting,
    but it was the ADMM of \cite{wei2013convergence} the first one shown to converge at the same rate as the centralized ADMM. 
    The algorithm of \cite{xu2017convergence} tracks the average gradients to converge to the exact optimum instead of just a neighborhood around it, as many algorithms back then. 
    The algorithm of \cite{pu2020push} can be used on top of directed graphs, which impose additional challenges. 
    A key novelty of our scheme, compared to this line of work, is that we consider the possibility of letting the nodes \emph{choose smartly} the neighbor to contact in order to make convergence faster.

    Work \cite{verma2021max} is, to the best of our knowledge, the only work similarly considering smart neighbor selection. 
    The authors propose Max-Gossip, a version of \cite{nedic2007rate} where the activated node averages its local (primal) parameter with that of the neighbor with whom the parameter difference in the largest. 
    They consider convex scalar functions $f_i$, and use Lyapunov analysis to prove convergence to an optimal value. 
    In contrast, here we obtain linear convergence \emph{rates} for smooth and strongly convex $f_i$ using duality theory.
    
    Moreover, our theorems extend the results in \cite{nutini2015coordinate}, where the GS rule was shown to be up to $d$ times faster than uniform sampling for $f{:}\ \Rbb^d {\ram} \Rbb$, to the case where this choice is constrained to a subset of the coordinates only, sets have different sizes, each coordinate belongs to exactly two sets, and sets activate uniformly at random. 
    This matches not only the decentralized case, but also parallel distributed settings such as
    \cite{tsitsiklis1986distributed,peng2016arock,xiao2019dscovr}.
    For the latter, \cite{you2016asynchronous} also analyzed the GS applied to coordinate subsets, but their sets are disjoint, accessible by any worker, and they do not quantify the speedup of the method with respect to uniform sampling.

\section{Dual formulation} \label{sec:model_and_problems}
    
    In this section, we define the notation, obtain the dual problem of \eqref{eq:primal_constrained}, and analyze the properties of the dual objective function.  
    We will assume throughout that the functions $f_i$ are $M_i$-smooth and $\mu_i$-strongly convex:
    
    \begin{ceqn}
    \begin{align*}
        f_i(y) &\leq f_i(x) + \Abraces{\nabla f(x), y-x} + (M_i / 2) \normtwo{y-x}^2  \\ 
        f_i(y) &\geq f_i(x) + \Abraces{\nabla f(x), y-x} + (\mu_i / 2) \normtwo{y-x}^2. 
    \end{align*}     
    \end{ceqn}

    We define the concatenated primal and dual variables 
    $\theta = [\theta_1^T,\ldots,\theta_n^T]^T \in \Rbb^{nd}$ and 
    $\lambda = [\lambda_1^T,\ldots,\lambda_E^T]^T \in~\Rbb^{Ed}$, respectively.  
    The graph's incidence matrix $A \in \Rbb^{n \times E} $ has exactly one 1 and one -1 per column $\ell$, in the rows corresponding to nodes $i,j: \ell \equiv (i,j)$, and zeros elsewhere (the choice of sign for each node is irrelevant).
    We call $u_i \in \Rbb^n$ the vector that has 1 in entry $i$ and 0 elsewhere; 
    we define $e_\ell \in \Rbb^E$ analogously.  
    We use $k \in [K]$ to indicate $k=1,\ldots,K$. 
    Vectors $\ones$ and $\zeros$ are respectively the all-one and all-zero vectors, and $I_d$ is the $d \times d$ identity matrix. 
    Finally, we define the block arrays $\Lambda = A \otimes I_d \in \Rbb^{nd \times Ed}$ and 
    $U_i = u_i \otimes I_d \in \Rbb^{nd \times d}$, where $\otimes$ is the Kronecker~product.
    
    We can rewrite now \eqref{eq:problem-b} as $\Lambda^T \theta = \zeros$, and the node variables as $\theta_i = U_i^T \theta$.
    The minimum value of \eqref{eq:primal_constrained} satisfies:
    \begin{ceqn}
    \begin{align*} 
    \underset{\theta: \Lambda^T \theta = \zeros}{\inf}  & \sum_{i=1}^n f_i(U_i^T \theta)
    \stackrel{(\text{a})}{=} \inf_{\theta} \sup_{\lambda} \Sbraces{ \sum_{i=1}^n f_i (U_i^T \theta) - \lambda^T \Lambda^T \theta } \\
    &\stackrel{(\text{b})}{=} \sup_{\lambda} \inf_{\theta} \Sbraces{ \sum_{i=1}^n f_i (U_i^T \theta) - \lambda^T \Lambda^T \theta } \\
    &= -\inf_{\lambda} \sup_{\theta} \sum_{i=1}^n \Sbraces{ (U_i^T \Lambda \lambda)^T U_i^T \theta -  f_i (U_i^T \theta) } \\
    &= -\inf_{\lambda} \sum_{i=1}^n f_i^* (U_i^T \Lambda \lambda) 
    \triangleq  -\inf_{\lambda} F(\lambda), \numberthis \label{eq:dual_problem}
    \end{align*}
    \end{ceqn}
    
    \noindent
    where (a) holds due to Lagrange duality and (b) holds by strong duality (see e.g. Sec. 5.4 in \cite{boyd2004convex}).
    Functions $f_i^*$ are the Fenchel conjugates of the $f_i$, and are defined as 
    
    \begin{ceqn}
    \[ f_i^*(y) = \sup_{x \in \Rbb^d} \paren{ y^T x - f_i(x) }. \]  
    \end{ceqn}
    
    Our set-wise CD algorithms converge to the optimal solution of \eqref{eq:primal_constrained} by solving \eqref{eq:dual_problem}. 
    In particular, they update a single dual variable 
    $\lambda_\ell,\ell \in [E]$ 
    at each iteration and converge to some minimum value $\lambda^*$ of $F(\lambda)$. 

    Since $\sum_{i=1}^n f_i(U_i^T \theta)$ in \eqref{eq:problem-a} is $M_{\max}$-smooth and $\mu_{\min}$-strongly convex in $\theta$, with $M_{\max} = \max_i M_i$ and $\mu_{\min} = \min_i \mu_i$, function $F$ is $L$-smooth with $L = \frac{\gamma_{\max}}{\mu_{\min}}$, where $\gamma_{\max}$ is the largest eigenvalue of $\LL$ (Sec. 4 in \cite{uribe2020dual}).
    We call $\gamma^+_{\min}$ the smallest non-zero eigenvalue\footnote{The ``+" stresses that $\gamma^+_{\min}$ is the smallest \emph{strictly positive} eigenvalue.} of $\LL$. 
    
    However, as shown next, function $F$ is \emph{not} strongly convex in the standard L2 norm, which is the property that usually facilitates obtaining linear rates in optimization  literature. 
    
    \begin{lemma} \label{lemma:no_SC}
        $F$ is not strongly convex in $\normtwo{\cdot}$. 
    \end{lemma}
    \begin{proof}
        Since $\Lambda$ does not have full column rank in the general case (i.e., unless the graph is a tree), there exist $w~\in~\Rbb^{Ed}$ such that $w \neq \zeros$ and $F(\lambda) = F(\lambda + tw) \; \forall t \in \Rbb$.
    \end{proof} 
    
    Nevertheless, we can still show linear rates for the set-wise CD algorithms using the following result.
    
    \begin{lemma}[\textbf{\textit{Appendix C of} \cite{hendrikx2019accelerated}}] \label{lemma:SC_in_AA}
        $F$ is $\sigma_A$-strongly convex in the semi-norm $\normAA{x} \triangleq (x^T \Lambda^+ \Lambda x)^\frac{1}{2}$, with $\sigma_A = \frac{\gamma^+_{\min}}{M_{\max}}$. 
    \end{lemma}
            
    Above, $\Lambda^+$ denotes the pseudo-inverse of $\Lambda$. 
    A key fact for the proofs in the next section is that matrix $\Lambda^+ \Lambda$ is a projector onto $\range(\Lambda^T)$, the column space of $\Lambda^T$.
    
    To simplify the notation, in what follows we assume that $d=1$, so that $\Lambda = A$, $U_i = u_i$, and the gradient
    $\nabla_\ell F(\lambda) = \frac{\partial F(\lambda)}{\partial \lambda_\ell}$ 
    of $F(\lambda)$ in the direction of $\lambda_\ell$ is a scalar. 
    In Sec. \ref{sec:d_larger_1} we discuss how to adapt our proofs to the case $d > 1$.

\section{Set-wise Coordinate Descent Algorithms} \label{sec:algos_and_conv}
    
    In this section we present the \emph{set-wise CD} algorithms, which can solve generic convex problems (and \eqref{eq:dual_problem} in particular)
    optimally and asynchronously.
    We analyze two possibilities for the coordinate choice within the accessible coordinate subset: (i) sampling uniformly at random (SU-CD), and (ii) applying the GS rule (SGS-CD).
    
    If coordinate $\ell$ is updated at iteration $k$ and assuming \hbox{$d=1$}, the standard CD update applied to $F(\lambda)$~is \cite{nesterov2012efficiency}:
    \begin{ceqn}
    \begin{equation} \label{eq:CD_update}
        \lambda^{k+1} = \lambda^k - \eta^k \nabla_\ell F(\lambda^k) e_\ell,
    \end{equation}
    \end{ceqn}
    where $\eta^k$ is the stepsize. Since $F(\lambda)$ is $L$-smooth, choosing $\eta^{k} = 1/L \; \forall k$ guarantees descent at each iteration \cite{nutini2015coordinate}:
    \begin{ceqn}
    \begin{equation} \label{eq:iter_progress}
        F(\lambda^{k+1}) \leq F(\lambda^k) - \frac{1}{2L} \paren{\nabla_\ell F(\lambda^k)}^2.
    \end{equation}
    \end{ceqn}
    
    Eq. \eqref{eq:iter_progress} will be the departure point to prove the linear convergence rates of SU-CD and SGS-CD. 
    
    We now define formally the set-wise CD algorithms. 
    
    \begin{definition}[\textit{\textbf{Set-wise CD algorithm}}] \label{def:set-wise_CD}
        In a set-wise CD algorithm, every coordinate 
        $\ell \in [E]$ is assigned to (potentially multiple) sets 
        $\Scal_i,i \in [n]$, such that all coordinates belong to at least one set.
        At any time, a set $\Scal_i$ may activate with uniform probability among the $i$; 
        the set-wise CD algorithm then chooses a single coordinate $\ell \in \Scal_i$ to update using \eqref{eq:CD_update}.
    \end{definition}
    
    The next remark shows how the decentralized problem \eqref{eq:primal_constrained} can be solved asynchronously with set-wise CD~algorithms.
    
    \begin{remark} \label{rem:set-wise_for_asych_optim}
        By letting (i) the $E$ coordinates\footnote{If $d{>}1$, the standard CD terminology calls each $\lambda_\ell$ a ``block coordinate", i.e. a vector of $d$ coordinates out of the $E \cdot d$ coordinates of function $F(\lambda)$.} 
        in Definition \ref{def:set-wise_CD} be the dual variables $\lambda_\ell$, and (ii) the $\Scal_i, i \in [n]$ be the sets of dual variables corresponding to the edges that are connected to each node $i$, nodes can run a set-wise CD algorithm to solve \eqref{eq:dual_problem} (and thus, also \eqref{eq:primal_constrained}) asynchronously.
    \end{remark}
    
    In light of Remark \ref{rem:set-wise_for_asych_optim}, in the following we illustrate the steps that should be performed by the nodes to run the set-wise CD algorithms to find a $\lambda^*$. 
    We first note that the gradient of $F(\lambda)$ in the direction\footnote{This is equivalent to saying ``the $\ell$-th (block) entry of the gradient $\nabla F(\lambda)$".} of $\lambda_\ell$ for $\ell \equiv (i,j)$ is 
    \begin{ceqn}
    \begin{equation} \label{eq:coord_gradient}
        \nabla_\ell F(\lambda) = A_{i\ell} \nabla f_i^*(u_i^T A \lambda) 
        + A_{j\ell} \nabla f_j^*(u_j^T A \lambda).
    \end{equation}
    \end{ceqn}

    Nodes can use \eqref{eq:CD_update} and \eqref{eq:coord_gradient} to update the $\lambda_\ell$ that they have access to (i.e., those corresponding to the edges they are connected to) as follows: 
    each node $i$ keeps in memory the current values of $\lambda_\ell, \ell \in \Scal_i$, which are needed to compute $\nabla f_i^*(u_i^T A \lambda)$.
    Then, when edge $\ell \equiv (i,j)$ is updated (either because node $i$ activated and contacted $j$, or vice versa), both $i$ and $j$ compute their respective terms in the right hand side of \eqref{eq:coord_gradient} and exchange them through their link. 
    Finally, both nodes compute \eqref{eq:coord_gradient} and update their copy of $\lambda_\ell$ applying \eqref{eq:CD_update}. 
    
    Algorithms \ref{alg:SU-CD} and \ref{alg:SGS-CD} below detail these steps for SU-CD and SGS-CD, respectively.
    We have used $\Ncal_i$ to indicate the set of neighbors of node $i$ (note that $\Scal_i = \{ \ell: \ell \equiv (i,j), j \in \Ncal_i \}$).
    Table \ref{tab:set_deifnitions} shows this and other set-related notation that will be frequently used in the sections that follow.

    We now proceed to describe the SU-CD and SGS-CD algorithms in detail, and prove their linear convergence rates.

    %
    %
    \begin{table}[t]
    \centering
        \caption{Set-related definitions} \vspace{-0.8em}
        \label{tab:set_deifnitions}
    \begin{tabu} {|ll|}  \hline
        $\Scal_i$       & Set of edges connected to node $i$ \\  
        $\Ncal_i$       & Set of neighbors of node $i$   \\ 
        $N_i$           & Degree of node $i$, i.e. $N_i = \abs{\Scal_i} = \abs{\Ncal_i}$ \\ 
        $N_{\max}$      & Maximum degree in the network, i.e. $\max_i N_i$ \\ 
        $T_i$           & Selector matrix of set $\Scal_i$ (see Definition 
        \ref{def:selector_matrices}) \\
        $\Scal_i'$      & Subset $\Scal_i' \subseteq \Scal_i$ such that 
                        $\Scal_i' \cap \Scal_j' = \emptyset$ if $i \neq j$ \\ 
        $T_i'$          & Selector matrix of set $\Scal_i$ \\ 
        \rule{0pt}{10pt}$\overline{\Scal_i'}$   
                        & Complement set of $\Scal_i'$ such that 
                        $\overline{\Scal_i'} = \Scal_i \setminus \Scal_i'$ \\ 
        \rule{0pt}{10pt}$\overline{T_i'}$       
                        & Selector matrix of set $\overline{\Scal_i'}$ \vspace{0.2em} \\ 
        \hline
    \end{tabu}
    \end{table}

    %
    %

    \subsection{Set-wise Uniform CD (SU-CD)}
    
        In SU-CD, the activated node chooses the neighbor uniformly at random, as shown in Alg. \ref{alg:SU-CD}. 
        We can compute the per-iteration progress of SU-CD taking expectation in \eqref{eq:iter_progress}:
        \begin{ceqn}
        \begin{align*} 
            \Ebb \Big[ F(\lambda^{k+1}) \mid \lambda^k & \Big]
            \leq F(\lambda^k) - \frac{1}{2L} \Exp{\paren{\nabla_\ell F(\lambda^k)}^2 \mid \lambda^k } \\
            &= F(\lambda^k) - \frac{1}{2Ln} \sum_{i=1}^n \frac{1}{N_i} \sum_{\ell \in \Scal_i}  
            \paren{\nabla_\ell F(\lambda^k)}^2 \\
            &\leq F(\lambda^k) - \frac{1}{L n N_{\max}} \normtwo{\nabla F(\lambda^k)}^2 
            \numberthis \label{eq:progress_SU-CD}
        \end{align*}
        \end{ceqn}
        where $N_i = \abs{\Scal_i}$ and $N_{\max} = \max_i N_i$.
        
        The standard procedure to show the linear convergence of CD in the centralized case is to lower-bound $\normtwo{\nabla F(\lambda)}^2$ using the strong convexity of the function \cite{nesterov2012efficiency, nutini2015coordinate}.
        However, since $F$ is \emph{not} strongly convex (Lemma \ref{lemma:no_SC}), we cannot apply this procedure to get the linear rate of SU-CD.
        
        We can, however, use $F$'s strong convexity in $\normAA{\cdot}$ instead (Lemma \ref{lemma:SC_in_AA}).
        The next result gives the core of the proof.

        \begin{lemma} \label{lemma:equality_of_norms}
            It holds that 
            \begin{ceqn}
            \begin{equation*}
                \normtwo{\nabla F(\lambda)} = \normAA{\nabla F(\lambda)} = \normAA{\nabla F(\lambda)}^*,
            \end{equation*}
            \end{ceqn}
            where $\normAA{\cdot}^*$ is the dual norm of $\normAA{\cdot}$, defined as (e.g. \cite{boyd2004convex})
            \begin{ceqn}
            \begin{equation} \label{eq:dual_norm_AA}
                \normAA{z}^* = \sup_{x \in \Rbb^d} \Cbraces{ z^T x \biggr\rvert \normAA{x} \leq 1 }.
            \end{equation}
            \end{ceqn}
        \end{lemma}
        
        \begin{proof}
            Note that 
            $\forall w \neq \zeros$ such that $F(\lambda + tw) = F(\lambda) \; \forall t$,
            it holds that $w^T \nabla F(\lambda) = 0$ 
            and thus $\nabla F(\lambda) \in \range(A^T)$.
            This means that $\AA \nabla F(\lambda) = I_E \nabla F(\lambda)$, and therefore it holds that
            $\normAA{\nabla F(\lambda)} = \normtwo{\nabla F(\lambda)}$.
            Finally, since the dual norm of the L2 norm is again the L2 norm, we have that also $\normAA{\nabla F(\lambda)}^* = \normtwo{\nabla F(\lambda)}$, which gives the result.
        \end{proof}
        
        We now use Lemma \ref{lemma:equality_of_norms} to prove the linear rate of SU-CD.

    %
    \setlength{\textfloatsep}{10pt} 
    \begin{algorithm}[t]
    \caption{Set-wise Uniform CD (SU-CD)}
    \label{alg:SU-CD}
    \begin{algorithmic}[1]
        \MYSTATE {\bfseries Input:} Functions $f_i$, step $\eta$, incidence matrix $A$, graph~$\Gcal$
        \MYSTATE Initialize $\theta_i^0, i=1,\ldots,n$ and $\lambda_\ell^0, \ell=1,\ldots,E$
        
        \MYSTATE \textbf{for} $k = 1,2,\ldots$ \textbf{do}
            \MYSTATE \algindent Sample activated node $i \in \{1,\ldots,n\}$ uniformly
            \MYSTATE \algindent Node $i$ picks neighbor $j \gets \Ucal \{h: h \in \Ncal_i \}$
            \MYSTATE \algindent Node $i$ computes $\nabla f_i^*(u_i^T A \lambda)$ and sends it to $j$
            \MYSTATE \algindent Node $j$ computes $\nabla f_j^*(u_j^T A \lambda)$ and sends it to $i$

        	\MYSTATE \algindent \parbox[t]{\dimexpr\linewidth-\algorithmicindent}{Nodes $i,j{:}\;(i,j)\equiv\ell$ use \eqref{eq:coord_gradient} to update their local copies of $\lambda_\ell$ by  
        	$\lambda_\ell^k \gets \lambda_\ell^{k-1} - \eta \nabla_\ell F(\lambda)$} \\
        	\MYSTATE \algindent $\lambda_m^k \gets \lambda_m^{k-1} \; \forall \text{ edges } m \neq \ell$
    \end{algorithmic}
    \end{algorithm}

        \begin{theorem}[\textit{\textbf{Rate of SU-CD}}] \label{theo:rate_SU-CD}
            SU-CD converges as
            \begin{ceqn}
            \[ \Exp{F(\lambda^{k+1}) \mid \lambda^k} - F(\lambda^*) \leq
            \paren{1 - \frac{2\sigma_A}{LnN_{\max}}} \Sbraces{F(\lambda^k) - F(\lambda^*)}. \]
            \end{ceqn}
        \end{theorem}
        
        \begin{proof}
            Since $F(\lambda)$ is strongly convex in $\normAA{\cdot}$ with strong convexity constant $\sigma_A$ (Lemma \ref{lemma:SC_in_AA}), it holds
            \begin{ceqn}
            \[ F(y) \geq F(x) + \Abraces{\nabla  F(x), y - x} + \frac{\sigma_A}{2} \normAA{y-x}^2. \]
            \end{ceqn}
            
            Minimizing both sides with respect to $y$ as in \cite{nutini2015coordinate} we get 
            
            \begin{ceqn}
            \begin{equation} \label{eq:guarantee_suboptim_SU}
                F(x^*) \geq F(x) - \frac{1}{2\sigma_A} \paren{ \normAA{\nabla  F(x)}^* }^2,
            \end{equation}
            \end{ceqn}
            
            \noindent
            and rearranging terms we can lower-bound $\paren{\normAA{\nabla F(x)}^*}^2$.
            
            Finally, we can use Lemma \ref{lemma:equality_of_norms} to replace 
            $\paren{\normtwo{\nabla  F(x)}}^2$ with $\paren{\normAA{\nabla F(x)}^*}^2$
            in \eqref{eq:progress_SU-CD}, and use the lower bound on $\paren{\normAA{\nabla F(x)}^*}^2$ given by \eqref{eq:guarantee_suboptim_SU} to get the result.
        \end{proof}
        
        Note that vector $\lambda$ has $\frac{1}{2} \sum_i N_i = E \leq \frac{nN_{\max}}{2}$ coordinates, where the inequality holds with equality for regular graphs. 
        We make the following remark.
        
        \begin{remark}
            If $\Gcal$ is regular, the linear convergence rate of SU-CD is $\frac{\sigma_A}{LE}$, which matches the rate of centralized uniform CD for strongly convex functions \cite{nesterov2012efficiency, nutini2015coordinate}, with the only difference that now the strong convexity constant $\sigma_A$ is defined over norm $\normAA{\cdot}$.
        \end{remark}
        
        In the next section we analyze SGS-CD and show that its convergence rate can be up to $N_{\max}$ times that of SU-CD.

    %
    %

    \subsection{Set-wise Gauss-Southwell CD (SGS-CD)}

    %
    \setlength{\textfloatsep}{10pt} 
    \begin{algorithm}[t]
    \caption{Set-wise Gauss-Southwell CD (SGS-CD)}
    \label{alg:SGS-CD}
    \begin{algorithmic}[1]
        \MYSTATE {\bfseries Input:} Functions $f_i$, step $\eta$, incidence matrix $A$, graph~$\Gcal$
        \MYSTATE Initialize $\theta_i^0, i=1,\ldots,n$ and $\lambda_\ell^0, \ell=1,\ldots,E$
        
        \MYSTATE \textbf{for} $k = 1,2,\ldots$ \textbf{do}
            \MYSTATE \algindent Sample activated node $i \in \{1,\ldots,n\}$ uniformly

            \MYSTATE \algindent All $h \in \Ncal_i$ compute $\nabla f_h^*(u_h^T A \lambda)$ and send it to $i$ \label{line:all_neighbors_send_grads}

            \MYSTATE \algindent Node $i$ computes $\nabla f_i^*(u_i^T A \lambda)$
            \MYSTATE \algindent \parbox[t]{\dimexpr\linewidth-\algorithmicindent}{Compute $\nabla_\ell F(\lambda) \; \forall \ell \in \Scal_i$ (equivalently, $\forall h \in \Ncal_i)$ with 
            \eqref{eq:coord_gradient} using the received $\nabla f_h^*(u_h^T A \lambda)$} \\
            
            \MYSTATE \algindent Node $i$ selects $j \gets \max_{h \in \Ncal_i} \abs{\nabla_\ell F(\lambda)}, 
            \ell {\equiv} (i,h)$

        	\MYSTATE \algindent Node $i$ sends $\nabla f_i^*(u_i^T A \lambda)$ to $j$
        	
        	\MYSTATE \algindent \parbox[t]{\dimexpr\linewidth-\algorithmicindent}{Nodes $i,j{:}\;(i,j)\equiv\ell$ use \eqref{eq:coord_gradient} to update their local copies of $\lambda_\ell$ by  
        	$\lambda_\ell^k \gets \lambda_\ell^{k-1} - \eta \nabla_\ell F(\lambda)$} \\
        	\MYSTATE \algindent $\lambda_m^k \gets \lambda_m^{k-1} \; \forall \text{ edges } m \neq \ell$
    \end{algorithmic}
    \end{algorithm}

    In SGS-CD, as shown in Alg. \ref{alg:SGS-CD}, the activated node $i$ selects the neighbor $j$ to contact applying the GS rule within the edges in~$\Scal_i$:
    \begin{ceqn}
    \begin{equation*}
        \ell = \argmax_{m \in \Scal_i} \paren{\nabla_m F(\lambda)}^2, 
    \end{equation*} 
    \end{ceqn}
    \noindent
    and then $j$ satisfies $\ell \equiv (i,j)$.
    In order to make this choice, all nodes $h \in \Ncal_i$ must send their $\nabla f_h^*(u_h^T A \lambda)$ to node $i$ (line \ref{line:all_neighbors_send_grads} in Alg. \ref{alg:SGS-CD}). 
    We discuss this additional communication step of SGS-CD with respect to SU-CD in Sec. \ref{sec:conclusion}.
    
    To obtain the convergence rate of SGS-CD we will follow the steps taken for SU-CD in the proof of Theorem \ref{theo:rate_SU-CD}.
    As done for SU-CD, we start by computing the per-iteration progress taking expectation in \eqref{eq:iter_progress} for SGS-CD:
    
    \begin{ceqn}
    \begin{equation} \label{eq:progress_SGS-CD}
        \Exp{F(\lambda^{k+1}) \mid \lambda^k} \leq 
        F(\lambda^k) - \frac{1}{2Ln} \sum_{i=1}^n \max_{\ell \in \Scal_i} \paren{\nabla_\ell F(\lambda^k)}^2. \hspace{-4pt}
    \end{equation} 
    \end{ceqn}
    
    Given this per-iteration progress, to proceed as we did for SU-CD we need to show 
    (i) that the sum on the right hand side of \eqref{eq:progress_SGS-CD} defines a norm, 
    and (ii) that strong convexity holds in its dual norm. 
    We start by defining the selector matrices $T_i$, which will significantly simplify notation.

    \begin{definition}[\textit{\textbf{Selector matrices}}] \label{def:selector_matrices}
        The selector matrices $T_i \in \Cbraces{0,1}^{N_i \times E}, i=1,\ldots,n$ select the coordinates of a vector in $\Rbb^E$ that belong to set $\Scal_i$. 
        Note that any vertical stack of the unitary vectors $\Cbraces{e_\ell^T}_{\ell \in \Scal_i}$ gives a valid $T_i$.
    \end{definition}
    
    We can now show that the sum in \eqref{eq:progress_SGS-CD} is a (squared) norm. 
    Since the operation involves applying $\max(\cdot)$ within each set $\Scal_i$, we will denote this norm $\normSM{x}$, where the subscript SM stands for ``Set-Max".
    
    \begin{lemma} 
        The function $\normSM{x}~\triangleq~\sqrt{\sum_{i=1}^n \norm{T_i x}_\infty^2}  
        = \sqrt{\sum_{i=1}^n \max_{j \in \Scal_i} x_j^2 } $ is a norm in $\Rbb^E$.
    \end{lemma}
    
    \begin{proof}
        Using $\max_{j \in \Scal_i} \paren{x_j^2 {+} y_j^2} {\leq} \max_{j \in \Scal_i} x_j^2 {+} \max_{j \in \Scal_i} y_j^2$ 
        and $\sqrt{a+b} \leq \sqrt{a} + \sqrt{b}$ we can show that $\normSM{\cdot}$ satisfies the triangle inequality.
        It is straightforward to show that $\normSM{\alpha x} = \abs{\alpha} \normSM{x}$ and 
        $\normSM{x} = 0$ iff $x = \zeros$.
    \end{proof}
    
    Following the proof of Theorem \ref{theo:rate_SU-CD}, we would like to show that $F$ is strongly convex in the dual norm $\normSM{\cdot}^*$. 
    Furthermore, we would like to compare the strong convexity constant $\sigma_\SM$ with $\sigma_A$ to quantify the speedup of SGS-CD with respect to SU-CD. 
    It turns out, though, that computing $\normSM{\cdot}^*$ is not easy at all; the main difficulty stems from the fact that sets $\Scal_i$ are overlapping (or non-disjoint), since each coordinate $\ell \equiv (i,j)$ belongs to both $\Scal_i$ and $\Scal_j$. The first scheme in Figure \ref{fig:example_sets} illustrates this fact for the 3-node clique. 

    \begin{figure}[t]
        \centering
        \includegraphics[trim={13cm 6cm 32cm 7cm}, clip, width=0.6\linewidth]{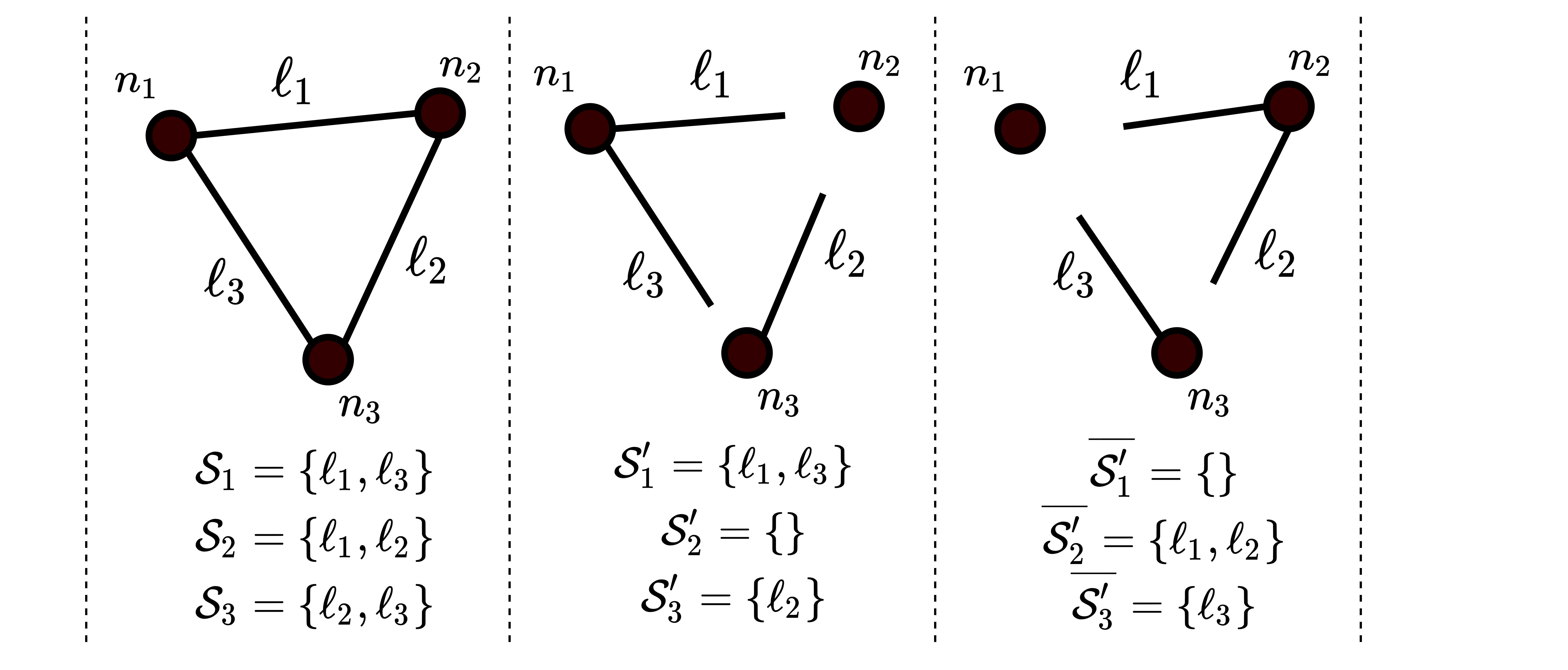}
        \caption{Example of sets $\Scal_i$ and \emph{one possibility} for $\Scal_i'$ and $\overline{\Scal_i'}$} 
        \label{fig:example_sets}
    \end{figure}

    To circumvent this issue, we define a new norm  $\normSMNO{\cdot}^*$ (``Set-Max Non-Overlapping") that we can directly relate to  $\normSM{\cdot}^*$ (Lemma \ref{lemma:ineq_SM_SMNO}) and whose value we can compute explicitly (Lemma \ref{lemma:value_norm_SMNO}), which will later allow us to relate its strong convexity constant $\sigma_\SMNO$ to $\sigma_A$ (Theorem \ref{theo:rate_SGS-CD}).

    \begin{definition}[\textit{\textbf{Norm $\normSMNO{\cdot}^*$}}]
        We assume that each coordinate $\ell \equiv (i,j)$ is assigned to only one of the sets $\Scal_i' \subseteq \Scal_i$ or
        $\Scal_j' \subseteq \Scal_j$, such that the new sets $\Cbraces{\Scal_i'}_{i=1}^n$ are non-overlapping (some sets can be empty), and all coordinates $\ell$ belong to exactly one set in $\Cbraces{\Scal_i'}$. 
        We name the selector matrices of these new sets $T_i'$, so that each possible choice of $\Cbraces{\Scal_i'}$ defines a different set $\{T_i'\}$.
        Then, we define
        \begin{ceqn}
        \begin{equation} \label{eq:def_norm_SMNO}
             \normSMNO{z}^* = \sup_{x} \Cbraces{ z^T x \; \biggr\rvert 
             \sqrt{ \sum_{i=1}^n \norm{T_i^* x}_\infty^2} \leq 1 },
        \end{equation}
        \end{ceqn}
        with the choice of non-overlapping sets
        \begin{ceqn}
        \begin{equation} \label{eq:optimal_Ti}
             \{T_i^*\} = 
             \arg \max_{ \{T_i'\} } \; \sum_{i=1}^n \norm{T_i' x}_\infty^2.
        \end{equation}
        \end{ceqn}
        
        Note that the maximizations in \eqref{eq:def_norm_SMNO} and \eqref{eq:optimal_Ti} are coupled. 
        We denote the value of $x$ that attains \eqref{eq:def_norm_SMNO} as $x_{\SMNO}^*$.
    \end{definition}
        
    The definition of sets $\Scal_i'$ corresponds to assigning each edge $\ell$ to one of the two nodes at its endpoints, as illustrated in the second scheme of Figure~\ref{fig:example_sets}. 
    Therefore, for each possible pair $\paren{\Cbraces{\Scal_h'}, \Cbraces{T_h'}}$ we can define a complementary pair $(\{\overline{\Scal_h'}\}, \{\overline{T_h'}\})$ 
    such that if $\ell \equiv (i,j)$ was assigned to $\Scal_i'$ in $\Cbraces{\Scal_h'}$, then it is assigned to $\overline{\Scal_j'}$ in $\{\overline{\Scal_h'}\}$.
    This corresponds to assigning $\ell$ to the opposite endpoint (node) to the one originally chosen, as shown in the third scheme of Figure~\ref{fig:example_sets}.
    With these definitions, it holds (potentially with some permutation of the rows):
    \begin{ceqn}
    \[ T_i = \begin{bmatrix} T_i' \\ \overline{T_i'} \end{bmatrix} =
        \begin{bmatrix} T_i' \\ \zeros \end{bmatrix} + 
        \begin{bmatrix} \zeros \\ \overline{T_i'} \end{bmatrix}, \; i=1,\ldots,n. \]
    \end{ceqn}
    
    We remark that the equality above holds for \emph{any} $\{T_i'\}$ corresponding to a feasible  assignment $\Cbraces{\Scal_i'}$, and in particular it hols for $(\Cbraces{\Scal_i^*}, \{T_i^*\})$.   
    This fact is used in the proof of the following lemma, which relates norms $\normSM{\cdot}^*$ and $\normSMNO{\cdot}^*$. 
    This will allow us to complete the analysis with $\normSMNO{\cdot}^*$, which we can compute explicitly (Lemma \ref{lemma:value_norm_SMNO}). 
    
    \begin{lemma} \label{lemma:ineq_SM_SMNO}
        The value of the dual norm of $\normSM{\cdot}$, denoted $\normSM{\cdot}^*$, satisfies
        $\paren{\normSM{\cdot}^*}^2 \geq \frac{1}{2} \paren{\normSMNO{\cdot}^*}^2$.
    \end{lemma}
    
    \begin{proof}
        By definition 
        \begin{ceqn}
        \[ \normSM{z}^* = \sup_x \Cbraces{ z^T x \; \biggr\rvert \sqrt{ \sum_{i=1}^n \norm{T_i x}_\infty^2 } \leq 1 }. \]
        \end{ceqn}
        
        By inspection we can tell that the $x$ that attains the supremum, denoted $x_{\SM}^*$, will satisfy 
        $\sum_{i=1}^n \norm{T_i x_{\SM}^*}_\infty^2 = 1$.
        We note now that 
        \begin{ceqn}
        \begin{equation} \label{eq:ineq_Ti}
            \sum_{i=1}^n \norm{T_i x}_\infty^2 
            = \sum_{i=1}^n \norm{\begin{bmatrix} T_i' \zeros \end{bmatrix} x + 
            \begin{bmatrix} \zeros \\ \overline{T_i'} \end{bmatrix} x}_\infty^2 
            \leq \sum_{i=1}^n \norm{T_i' x}_\infty^2 + \sum_{i=1}^n \norm{\overline{T_i'} x}_\infty^2
            \leq 2 \sum_{i=1}^n \norm{\widehat{T_i'} x}_\infty^2,
        \end{equation}
        \end{ceqn}
        with 
        \begin{ceqn}
        \begin{equation} \label{eq:max_over_Ti}
            \{\widehat{T_i'}\} = \arg\max_{\{T_i'\},\{\overline{T_i'}\}} 
        \paren{\sum_{i=1}^n \norm{T_i' x}_\infty^2, \sum_{i=1}^n \norm{\overline{T_i'} x}_\infty^2}.
        \end{equation}
        \end{ceqn}

        Note that if we evaluate \eqref{eq:max_over_Ti} at $x_{\SMNO}^*$, due to \eqref{eq:optimal_Ti} we have $\{\widehat{T_i'}\} = \{T_i^*\}$.
        Also, by inspection of problem \eqref{eq:def_norm_SMNO} we know that $x_{\SMNO}^*$ satisfies
        $\sum_{i=1}^n \norm{T_i^* x_{\SMNO}^*}_\infty^2 = 1$. 
        Therefore, \eqref{eq:ineq_Ti} says
        \begin{ceqn}
        \[ \frac{1}{2} \sum_{i=1}^n \norm{T_i x_{\SMNO}^*}_\infty^2 = 
        \sum_{i=1}^n \norm{T_i \frac{x_{\SMNO}^*}{\sqrt{2}}}_\infty^2 \leq 1, \]
        \end{ceqn}
        from where we conclude that coordinate-wise it must hold 
        $x_{\SM}^*~\succeq~\frac{1}{\sqrt{2}} x_{\SMNO}^*$,
        and thus $\normSM{z}^* \geq \frac{1}{\sqrt{2}} \normSMNO{z}^*$.
    \end{proof}
    
    The next lemma gives the value of $\normSMNO{x}^*$ explicitly, which will be needed to compare the strong convexity constant $\sigma_{\SMNO}$ with $\sigma_A$.
    
    \begin{lemma} \label{lemma:value_norm_SMNO}
        It holds that $\normSMNO{x}^* = \sqrt{\sum_{i=1}^n \norm{T_i^* x}_1^2}$.
    \end{lemma}
    
    \begin{proof}
        Since the sets $\{\Scal_i^*\}$ are non-overlapping and in \eqref{eq:def_norm_SMNO} norm $\norm{\cdot}_\infty$ is applied per-set, the entries $x_\ell$ of $x_{\SMNO}^*$ will have $\abs{x_\ell} = x^{(i)} \geq 0 \; \forall \; \ell \in \Scal_i^*$ and the sign will match that of the entries of $z$, i.e. $\sign(x_\ell) = \sign(z_\ell)$. 
        The maximization of \eqref{eq:def_norm_SMNO} then becomes 
        \begin{ceqn}
        \begin{equation*}
            \begin{aligned}
            & \maximize_{\Cbraces{x^{(i)}}} && \sum_{i=1}^n \sum_{\ell \in \Scal_i^*}  \paren{ \abs{z_\ell} \cdot x^{(i)} } \\
            & \st && \sqrt{\sum_{i=1}^n \paren{x^{(i)}}^2} \leq 1.
            \end{aligned}
        \end{equation*}
        \end{ceqn}
        
        Factoring out $x^{(i)}$ in the objective and noting that 
        $\sum_{\ell \in \Scal_i^*} \abs{z_\ell} = \norm{T_i^* z}_1$, we can define 
        $w = [x^{(1)}, \ldots, x^{(n)}]^T$ and $y = \Sbraces{ \norm{T_1^* z}_1, \ldots, \norm{T_n^* z}_1 }^T$ 
        so that \eqref{eq:def_norm_SMNO} now reads
        \begin{ceqn}
        \[ \normSMNO{z}^* = \sup_w \Cbraces{ y^T w \; \biggr\rvert \normtwo{w} \leq 1}. \]
        \end{ceqn}
        
        The right hand side is the definition of $\normtwo{\cdot}^*$, the dual of the L2 norm, evaluated at $y$. Since $\normtwo{\cdot}^* = \normtwo{\cdot}$, we have that
        $ \normSMNO{z}^* = \normtwo{y} = \sqrt{\sum_{i=1}^n \norm{T_i' z}_1^2} $.        
    \end{proof}
    
    We can now prove the linear convergence rate of SGS-CD. 
    
    \begin{theorem}[\textit{\textbf{Rate of SGS-CD}}] \label{theo:rate_SGS-CD}
        SGS-CD converges as
        \begin{ceqn}
        \begin{equation*} 
            \Exp{F(\lambda^{k+1}) \mid \lambda^k} - F(\lambda^*) \leq \\
            \paren{1 - \frac{2 \sigma_{\SMNO}}{Ln}} \Sbraces{F(\lambda^k) - F(\lambda^*)}, 
        \end{equation*}
        \end{ceqn}
        with 
        \begin{ceqn}
        \begin{equation} \label{eq:ineqs_sA_sSMNO}
           \frac{\sigma_A}{N_{\max}} \leq \sigma_{\SMNO} \leq \sigma_A.  
        \end{equation}
        \end{ceqn}
    \end{theorem}

    \begin{proof}
        We start by proving \eqref{eq:ineqs_sA_sSMNO} by showing that strong convexity in $\normAA{\cdot}$ implies strong convexity in $\normSMNO{\cdot}^*$, which will give the inequalities as a by-product of the analysis.
        Below we assume that $x \in \range(A^T)$; the results here can then be directly applied to the proofs above because $\normAA{\cdot}, \normSM{\cdot}, \normSMNO{\cdot}$ and their duals are applied to $\nabla F(\lambda)$, which is always in $\range(A^T)$ (Lemma \ref{lemma:equality_of_norms}). 
    
        For $x \in \range(A^T)$ it holds that 
        (Lemmas \ref{lemma:equality_of_norms} and \ref{lemma:value_norm_SMNO}):
        \begin{ceqn}
        \[ \normAA{x}^2 = \normtwo{x}^2 = \sum_{i=1}^E x_i^2 = \sum_{i=1}^n \normtwo{T_i^* x}^2 
        \text{\qquad and \qquad}
             \paren{\normSMNO{x}^*}^2 = \sum_{i=1}^n \norm{T_i^* x}_1^2. \]
        \end{ceqn}

        We also note that, using the Cauchy-Schwarz inequality and denoting $[v]_i$ the $i$\ts{th} entry of vector $v$, it holds both that
        \begin{ceqn}
        \[\sum_{i=1}^n \normtwo{T_i^* x}^2 \leq 
            \sum_{i=1}^n \Bigg( \sum_{j \in \Scal_i^*} |x_j| \Bigg)^2 =
            \sum_{i=1}^n \norm{T_i^* x}_1^2 \]
        \end{ceqn}
        and
        \begin{ceqn}
        \[ \sum_{i=1}^n \norm{T_i^* x}_1^2 
            = \sum_{i=1}^n \paren{ \ones^T 
            \bigg[ \Big| [T_i^*x]_1 \Big| , \ldots, \Big| [T_i^*x]_{N_i^*} \Big| \bigg]^T }^2 \\
            \stackrel{\text{C.S.}}{\leq}
            \sum_{i=1}^n N_i^* \normtwo{T_i^*x}^2
            \leq N_{\max} \sum_{i=1}^n \normtwo{T_i^*x}^2, \]
        \end{ceqn}
        
        \noindent
        where $N_i^* = \abs{\Scal_i^*}$.
        We can summarize these relations as 
        
        \begin{ceqn}
        \[ \frac{1}{N_{\max}} \paren{\normSMNO{x}^*}^2
        \leq \normAA{x}^2 \leq  \paren{\normSMNO{x}^*}^2. \]
        \end{ceqn}
        
        Using these inequalities in the strong convexity definitions, similarly to \cite{nutini2015coordinate}, we get both
        \begin{ceqn}
        \begin{align} \label{eq:lower_ineq}
            F(y) &\geq F(x) {+} \Abraces{\nabla F(x), y-x} {+} \frac{\sigma_A}{2} \paren{\normAA{y-x}}^2 \\
            &\geq F(x) {+} \Abraces{\nabla F(x), y-x} {+} \frac{\sigma_A}{2N_{\max}} \paren{\normSMNO{y-x}^*}^2, 
        \end{align}
        \end{ceqn}
        
        and
        
        \begin{ceqn}
        \begin{align} \label{eq:higher_ineq}
            F(y) &\geq F(x) {+} \Abraces{\nabla F(x), y{-}x} {+} \frac{\sigma_{\SMNO}}{2} \paren{\normSMNO{y{-}x}^*}^2 \\
            &\geq F(x) {+} \Abraces{\nabla F(x), y-x} {+} \frac{\sigma_{\SMNO}}{2} \paren{\normAA{y-x}}^2.
        \end{align} 
        \end{ceqn}
        
        Equation \eqref{eq:lower_ineq} says that $F$ is at least $\frac{\sigma_A}{N_{\max}}$-strongly convex in $\normSMNO{\cdot}^*$, and eq. \eqref{eq:higher_ineq} says that $F$ is at least $\sigma_{\SMNO}$-strongly convex in $\normAA{\cdot}$.
        Together they imply \eqref{eq:ineqs_sA_sSMNO}. 
        
        To get the rate of SGS-CD, and following the procedure of SU-CD, we need to lower-bound the per-iteration progress $\frac{1}{2Ln} \normSM{\nabla F(\lambda)}^2$   
        in \eqref{eq:progress_SGS-CD}. 
        For this we will use the strong convexity in $\normSM{\cdot}^*$, which we can obtain from the strong convexity that we just proved for $\normSMNO{\cdot}^*$, as shown next. 
        
        Stating that $F$ is at least $\sigma_{\SM}$-strongly convex in $\normSM{\cdot}^*$ and using Lemma \ref{lemma:ineq_SM_SMNO} we obtain:
        \begin{ceqn}
        \begin{align} \label{eq:SC_in_SM}
            F(y) 
            &\geq F(x) + \Abraces{\nabla F(x), y-x} + \frac{\sigma_{\SM}}{2} \paren{\normSM{y-x}^*}^2 \\
            &\geq F(x) + \Abraces{\nabla F(x), y-x} + \frac{\sigma_{\SM}}{2} \frac{1}{2} \paren{\normSMNO{y-x}^*}^2, 
        \end{align}
        \end{ceqn}
        from where we conclude that $\sigma_{\SM} = 2 \sigma_{\SMNO}$.
        
        Minimizing both sides of the first inequality in \eqref{eq:SC_in_SM} with respect to $y$ we obtain 
        \begin{ceqn}
        \begin{equation} \label{eq:guarantee_suboptim_SM}
            F(x^*) \geq F(x) - \frac{1}{2\sigma_{\SM}} \paren{\normSM{\nabla F(x)}}^2,
        \end{equation}
        \end{ceqn}
        which is analogous to \eqref{eq:guarantee_suboptim_SU}, and rearranging terms gives a lower bound on $\normSM{\nabla F(\lambda)}^2$. 
        Using this lower bound in \eqref{eq:progress_SGS-CD} and replacing $\sigma_{\SM} = 2 \sigma_{\SMNO}$ gives the rate of SGS-CD.
    \end{proof}

    Theorem \ref{theo:rate_SGS-CD} states that SGS-CD can be up to $N_{\max}$ times faster than SU-CD.
    This result is analogous to that of \cite{nutini2015coordinate} for the GS rule compared to uniform sampling in centralized~CD.

    Although this maximum speedup is an upper bound and may not always be achievable, we can think of the following scenario where this gain is attained: let all sets have the same size $|\Scal_i| = N_{\max} \; \forall i$, exactly $m$ out~of the $N_{\max}$ coordinates in each set have $\nabla_m F(\lambda) = 0$, and only one $\ell$ have $\nabla_\ell F(\lambda) \neq 0$. 
    In this case, on average \emph{only $\frac{1}{N_{\max}}$ times will SU-CD choose the coordinate that gives some improvement}, while SGS-CD will do it at all iterations.

    Note that this example requires the gradients of all coordinates to be independent, which is not verified in the decentralized optimization setting: according to eq. \eqref{eq:coord_gradient}, for a $\nabla_m F$ to be zero, it must hold that $\nabla f_i^* = \nabla f_j^*, m \equiv (i,j)$. 
    But unless this equality holds for \emph{all} $(i,j) \in \Ecal$ (i.e., unless the minimum has been attained), $\lambda$ will continue changing, and the $\nabla f_i^*$ will differ.
    Thus, the gains of SGS-CD in this setting may not attain the upper bound.
    
    Nevertheless, when it comes to parallel distributed setups, the coordinates are not necessarily coupled as in the decentralized case, and thus the $N_{\max}$ speedup of SGS-CD is still achievable, as shown in our simulations below.

    \subsection{Case $d > 1$} \label{sec:d_larger_1}
    
    To extend the proofs above for $d~>~1$, the block~arrays $\Lambda$ and  $U_i$ should be used instead of $A$ and $u_i$, and the selector matrices $T_i$ should be redefined in the same way. 
    Then, all the operations that in the proofs above are applied \emph{per entry} (scalar coordinate) of the vector $\lambda$, should now be applied to \emph{the magnitude} of each vector (or ``block" \cite{nutini2017let}) coordinate $\lambda_\ell \in \Rbb^d$ of $\lambda \in \Rbb^{Ed}$. 
    Also, since $\nabla_m F \in \Rbb^d$, in this case the GS rule becomes 
    $\argmax_{m \in \Scal_i} \normtwo{\nabla_m F(\lambda)}^2$.

\section{Numerical Results} \label{sec:numerical_results}

    \begin{figure*}[t]
	\centering
	\begin{subfigure}{0.51\linewidth}
		\centering \includegraphics[trim={1mm 1mm 1mm 1mm},clip,width=\linewidth]{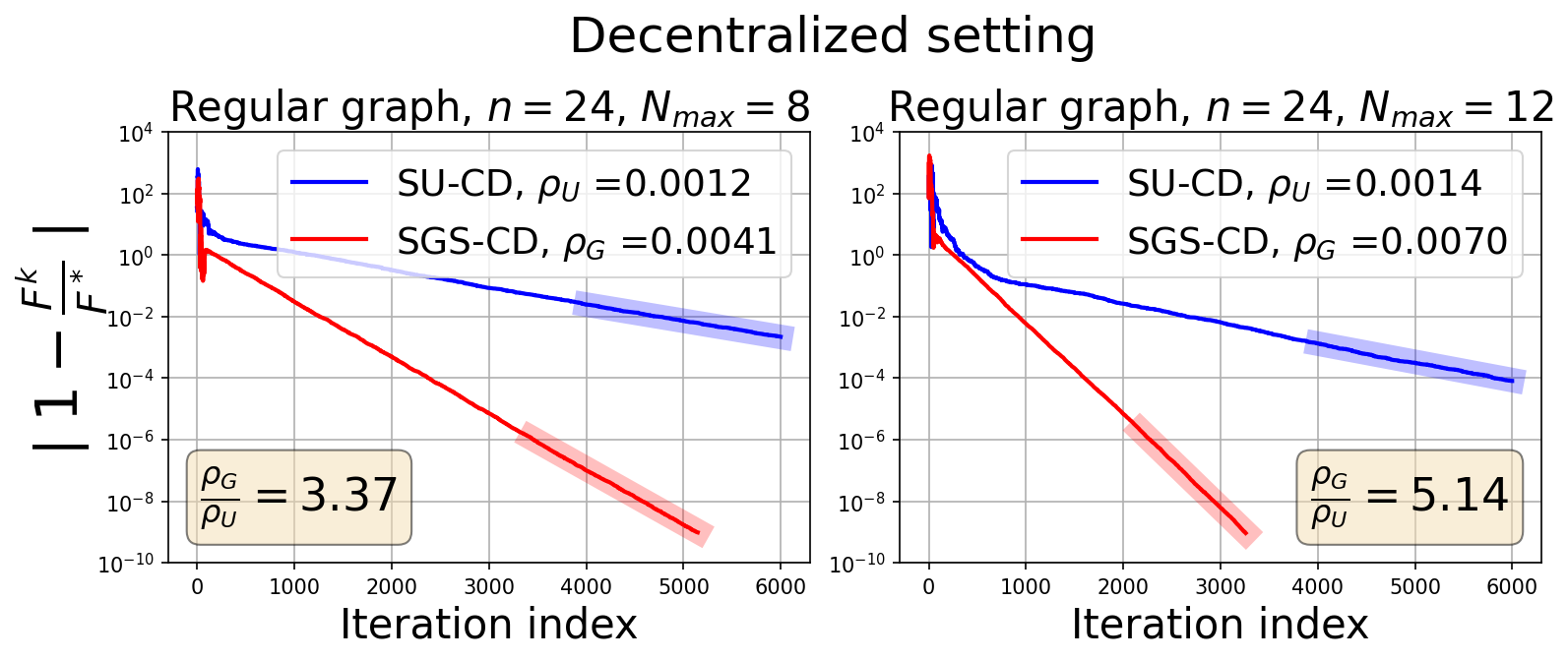}
	\end{subfigure}
  	\begin{subfigure}{0.48\linewidth}
    	\centering \includegraphics[trim={1mm 1mm 1mm 1mm},clip,width=\linewidth]{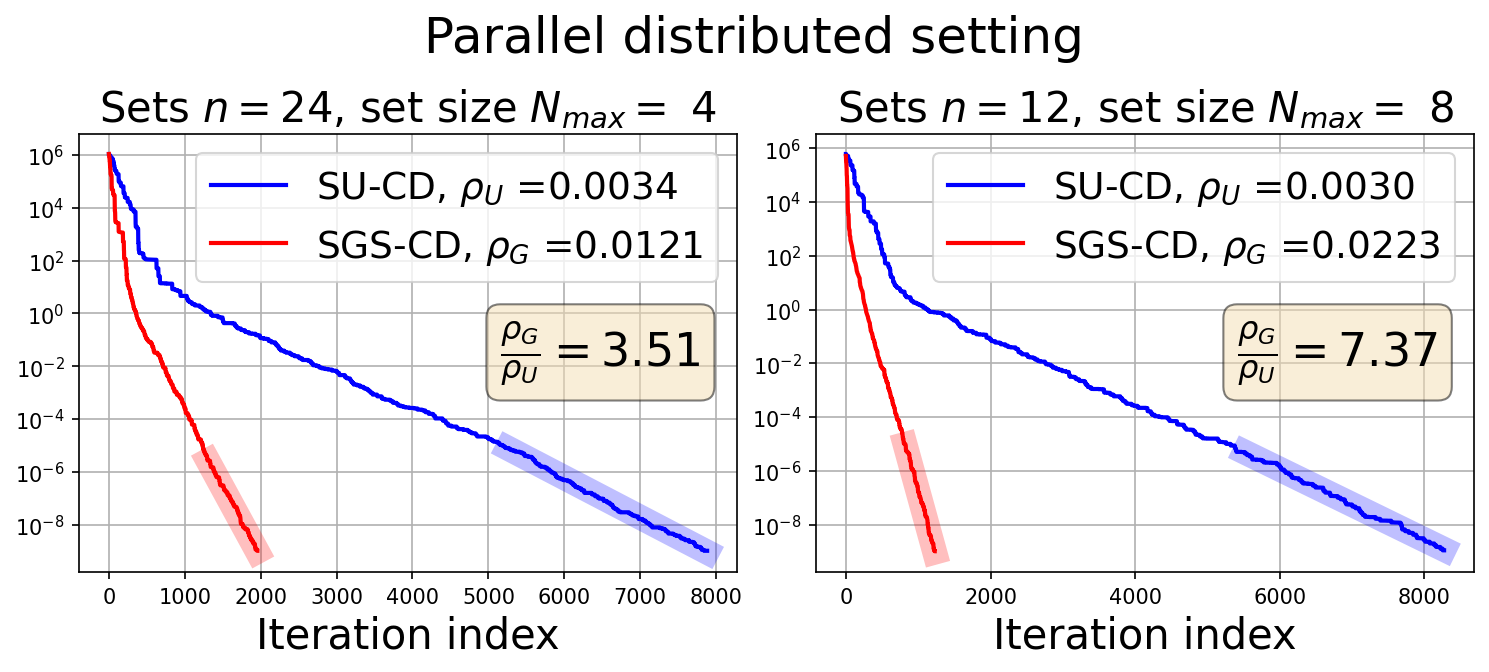}
  	\end{subfigure} 
    \caption{Comparison of the convergence rates of SU-CD and SGS-CD in two settings: decentralized optimization over a network (left plots), and parallel distributed computation with parameter server (right plots). 
    Given the linear suboptimality reduction 
    $F(\lambda^k)-F(\lambda^*) \leq (1-\rho)^k[F(\lambda^0)-F(\lambda^*)]$, 
    the thick transparent lines show the part of the curves used to estimate $\rho_U$ and $\rho_G$ for SU-CD and SGS-CD, respectively.
    The ratio 
    $\frac{\rho_G}{\rho_U}$ 
    increases notably with $N_{\max}$, in agreement with the theory.}
    \label{fig:simulation}
    \end{figure*}
    
    Figure \ref{fig:simulation} shows the remarkable speedup  of SGS-CD with respect to SU-CD in both the decentralized (left plots) and the parallel distributed (right plots) settings\footnote{The code is available at \url{https://github.com/m-costantini}}. 
    
    For the decentralized setting we created two regular graphs of $n=24$ nodes and degrees $N_{\max} = 8$ and 12, respectively. 
    The local functions were $f_i(\theta) = \theta^T c I_d \theta$ with $d = 5$, and $c=50$ if $(i$ modulo $N_{\max}) = 0$ and $c=1$ otherwise,~where $i$ is the index of each node.
    We chose these $f_i$ so that each node would have (approximately) one neighbor out of the $N_{\max}$ with whom the coordinate gradient would have maximum disagreement, thus maximizing the chances of observing differences between SU-CD and SGS-CD.
    
    For the parallel distributed setting, we created a problem that was separable per-coordinate, and we tried to recreate the conditions described in the previous section to approximate the $N_{\max}$ gain. 
    We chose $F(x) = x^T D x$ with $x \in \Rbb^d$ and $d = 48$. 
    Matrix $D$ was diagonal with its non-zero entries sampled from $\Ncal(10, 3)$. 
    We then created $n$ sets of $N_{\max}$ coordinates such that each coordinate belonged to exactly two sets, similarly to the parallel distributed scenario with parameter server where each worker has access to a subset of the coordinates only. 
    We simulated two different distributions of the $d=48$ coordinates: one with $n=24$ sets of $N_{\max} = 4$ coordinates each, and another with $n=12$ sets of $N_{\max} = 8$ coordinates each. 
    Following the reasoning in the previous section, we set the initial value of $(N_{\max}-1)$ coordinates in each set to $x_m^0 = 1$ (close to the optimal value $x_m^*=0$), and the one remaining to $x_\ell^0 = 100$ (far away from $x_\ell^*=0$).
    
    The plots in Figure \ref{fig:simulation} show the steep rate gain of SGS-CD with respect to SU-CD as $N_{\max}$ increases. 
    To quantify this gain we denoted $(1-\rho)$ the suboptimality reduction factor and we estimated $\rho_U, \rho_G$ for SU-CD and SGS-CD, respectively, from the last third of the suboptimality curves. 
    Theorem \ref{theo:rate_SGS-CD} says that $1 \leq \frac{\rho_G}{\rho_U} \leq N_{\max}$, and indeed, this is verified in both settings.
    In particular, for the decentralized setting $\frac{\rho_G}{\rho_U}$ is approximately in the middle of this range for both regular graphs.
    In the parallel distributed setting, however, the ratio is much closer to $N_{\max}$, as predicted.

\section{Discussion} \label{sec:conclusion}

    We have presented the class of \emph{set-wise CD} algorithms, where in a multi-agent system workers are allowed to modify only a subset of the total number of coordinates at each iteration. 
    These algorithms are suitable for asynchronous decentralized optimization and distributed parallel optimization.
    We studied specifically two set-wise CD variants, SU-CD and SGS-CD, which required developing a new methodology that extends previous results on CD methods. 
    
    We obtained the convergence rates of SU-CD and SGS-CD for smooth and strongly convex functions $f_i$ and showed that they are analogous, except for the network-related parameters, to those given in \cite{nutini2015coordinate} for their centralized counterparts. 
    More precisely, we showed that SGS-CD can be up to $N_{\max}$ (the size of the largest coordinate set) times faster than SU-CD; we further elaborated on the conditions under which such speedup may be attainable, and confirmed these predictions with numerical simulations.
    
    A limitation of SGS-CD with respect to SU-CD is that all the neighbors of the activated node $i$ must compute their $\nabla f_h^*$ and send it to node $i$ (line \ref{line:all_neighbors_send_grads} in Alg. \ref{alg:SGS-CD}).  
    This additional overhead with respect to SU-CD is analogous to that of the GS rule in centralized CD, which is the reason why GS makes sense only for problems with certain separability and sparsity structures \cite{nutini2015coordinate, nutini2017let}.
    Designing algorithms that approximate SGS-CD at the cost of SU-CD is a subject of future work.
    
    A possibility that was not accounted for in this study is letting the nodes use different stepsizes to update each $\lambda_\ell$ \cite{nutini2017let}. 
    Indeed, function $F$ is coordinate-wise smooth with constant $L_\ell = \paren{\frac{1}{\mu_i} + \frac{1}{\mu_j}}$ for $\ell \equiv (i,j)$; this could be used by each node to choose a different stepsize $\eta_\ell \geq \frac{1}{L}$ for each $\lambda_\ell$, which would make convergence faster. 
    Methods to estimate the per-coordinate smoothness when it is not known a priori were discussed in \cite{nesterov2012efficiency, nutini2017let}.
    
    Another way of obtaining faster convergence is to use Nesterov acceleration, as done in \cite{hendrikx2019accelerated}, now \emph{on top} of the smart neighbor choosing rule. 
    Although this would entail partially sacrificing the complete lack of coordination allowed by the set-wise CD algorithms presented here (because acceleration couples the coordinate updates), combined with per-coordinate specific stepsizes and well-designed neighbor sampling rules it would open the possibility of obtaining \emph{the fastest} decentralized set-wise CD algorithm, similarly to recent results for accelerated centralized CD \cite{nesterov2017efficiency, allen2016even}.

\bibliographystyle{ieeetr}
\bibliography{bibliography}

\end{document}